\documentclass[12pt,a4paper]{article}
\usepackage{graphicx}%
\usepackage{amsmath,amsthm,amsfonts,amssymb} 
\usepackage{txfonts,eucal} 

\DeclareMathOperator\Char{Char}%
\DeclareMathOperator\diag{diag}%
\newcommand{\li}{\langle}
\newcommand{\re}{\rangle}
\newcommand{\lire}{\li\,\cdot\,,\cdot\,\re}
\newcommand{\quer}{\overline{\phantom{a}}}
\newcommand{\ol}[1]{\overline{#1}}
\newcommand{\x}{\times} 
\newcommand{\ulwedge}{\mathbin{\underline\wedge}} 

\newcommand{\rep}{\mathrel{\parallel_\mathrm{\,R}}}
\newcommand{\lip}{\mathrel{\parallel_\mathrm{\,L}}}
\newcommand{\Cp}{\mathrel{\parallel_{\,C}}}
\newcommand{\Cpp}{\mathrel{\parallel_{\,C^\perp}}}
\newcommand{\Ccp}{\mathrel{\parallel_{\,C'}}}
\newcommand{\reS}{\cS_\mathrm{R}}
\newcommand{\liS}{\cS_\mathrm{L}}
\newcommand{\CS}{\cS_{C}}

\newcommand{\trenn}{{-\hspace{0pt}}}

\newcommand{\bPU}{\bP(U)} 
\newcommand{\bPV}{\bP(V)} 
\newcommand{\bPH}{\bP(H)} 
\newcommand{\bPVV}{{\bP(V\wedge V)}} 
\newcommand{\bPHH}{{\bP(H\wedge H)}} 
\newcommand{\bGV}{\bG_{1}(V)}
\newcommand{\cPV}{\cP(V)} 
\newcommand{\cLV}{\cL(V)} 
\newcommand{\cLH}{\cL(H)} 
\newcommand{\cPVV}{{\cP(V\wedge V)}} 

\renewcommand{\phi}{\varphi}
\newcommand{\cE}{{\mathcal E}}
\newcommand{\cF}{{\mathcal F}}
\newcommand{\cH}{{\mathcal H}}
\newcommand{\cL}{{\mathcal L}}
\newcommand{\cP}{{\mathcal P}}
\newcommand{\cQ}{{\mathcal Q}}
\newcommand{\cS}{{\mathcal S}}
\newcommand{\cR}{{\mathcal R}}

\newcommand{\bP}{{\mathbb P}}
\newcommand{\bG}{{\mathbb G}}

\newtheorem{thm}{Theorem}[section]
\newtheorem{prop}[thm]{Proposition}
\newtheorem{cor}[thm]{Corollary}
\newtheorem{lem}[thm]{Lemma}
\theoremstyle{definition}
\newtheorem{defn}[thm]{Definition}
\theoremstyle{remark} 
\newtheorem{rem}[thm]{Remark}
\numberwithin{equation}{section}

\renewcommand{\theenumi}{\alph{enumi}}%
\begin{document}


\author{Hans Havlicek}
\title{Clifford Parallelisms and\\ External Planes to the Klein quadric}
\date{\emph{Dedicated to the memory of G\"{u}nter Pickert}}

\maketitle

\begin{abstract}
For any three{\trenn}dimensional projective space $\bPV$, where $V$ is a vector
space over a field $F$ of arbitrary characteristic, we establish a one-one
correspondence between the Clifford parallelisms of $\bPV$ and those planes of
$\bPVV$ that are external to the Klein quadric representing the lines of
$\bPV$. We also give two characterisations of a Clifford parallelism of $\bPV$,
both of which avoid the ambient space of the Klein quadric.
\end{abstract}


\section{Introduction}\label{sect:intro}

A theory of \emph{Clifford parallelisms}\/ on the line set of a projective
space over a field $F$ was developed already several decades ago; see
\cite[pp.\,112--115]{john-03a} and \cite[\S\,14]{karz+k-88a} for a detailed
survey. This theory is based upon \emph{projective double spaces}\/ and their
algebraic description via left and right multiplication in two classes of
$F${\trenn}algebras: \eqref{A}~\emph{quaternion skew fields}\/ with centre $F$
and \eqref{B}~\emph{purely inseparable field extensions of $F$}\/ satisfying
some extra property. More precisely, this theory is built up from \emph{two}\/
parallelisms, hence the name ``double space''. These parallelisms are identical
in case \eqref{B} and distinct otherwise. It is common to distinguish between
them by adding the appropriate attribute ``left'' or ``right''.
\par
In the present paper we aim at giving characterisations of a \emph{single}\/
Clifford parallelism. We confine ourselves to the three{\trenn}dimensional
case. First, we collect the necessary background information in some detail, as
it is widespread over the literature. In Section~\ref{sect:external}, we
consider a three{\trenn}dimensional projective space $\bPV$ on a vector space
$V$ over a field $F$ (of arbitrary characteristic) and the Klein quadric $\cQ$
representing the lines of $\bPV$. Each \emph{external plane to the Klein
quadric}, \emph{i.\,e.}\/ a plane in the ambient space of $\cQ$ that has no
point in common with $\cQ$, is shown to give rise to a parallelism of $\bPV$.
Polarising such a plane with respect to the Klein quadric yields a second
plane, which is also external to $\cQ$, and hence a second parallelism. We
establish in Theorem~\ref{thm:doublespace} that these parallelisms turn $\bPV$
into a projective double space, which therefore can be described algebraically
in the way we sketched above. The main result in Section~\ref{sect:Cliff-Klein}
is Theorem~\ref{thm:C+C'}, where we reverse the previous construction. What
this all amounts to is a \emph{one-one correspondence between Clifford
parallelisms and planes that are external to the Klein quadric}. Finally, in
Section~\ref{sect:character} we present two characterisations of a Clifford
parallelism. Both of them avoid the ambient space of the Klein quadric, even
though our proofs rely on the aforementioned correspondence.
Theorem~\ref{thm:hyper} characterises a parallelism as being Clifford via the
property that any two distinct parallel classes are contained in a
\emph{geometric hyperplane}\/ of the Grassmann space formed by the lines of
$\bPV$. A second characterisation is given in Theorem~\ref{thm:crossed}; it
makes use of a presumably new criterion, namely what we call the
\emph{condition of crossed pencils}.
\par
When dealing with Clifford parallelisms in general, one must not disregard the
rich literature about the classical case over the real numbers. The recent
articles \cite{bett+r-12a} and \cite{cogl-15a} provide a detailed survey;
further sources can be found in \cite[p.\,10]{blasch-60a} and \cite{coxe-98a}.
Some of the classical results remain true in a more general setting, even
though not in all cases. For example, a quadric without real points will carry
two reguli after its complexification. Any such regulus can be used to
characterise a Clifford parallelism in the real case. By \cite{blunck+p+p-07a},
this result still applies (with some subtle modifications) in case~\eqref{A}
but, according to \cite{havl-15a}, it fails in case~\eqref{B}. A crucial
question is therefore to find classical results that can be generalised without
limitation. Many of our findings are of this kind, and we shall give references
to related work over the real numbers in the running text.
\par
Finally, for the sake of completeness, let us mention that
higher{\trenn}dimensional analogues of Clifford parallelisms can be found in
\cite{gier-82a} and \cite{tyrr+s-71a}.

\section{Preliminaries}\label{sect:prelim}

Let $U$ be a vector space over a field\footnote{We assume the multiplication in
a field to be commutative, in a skew field it may be commutative or not.} $F$.
It will be convenient to let the \emph{projective space}\/ $\bPU$ be the set of
all subspaces of $U$ with \emph{incidence}\/ being symmetrised inclusion. We
adopt the usual geometric terms: If $Z\subset U$ is a subspace of $U$ with
vector dimension $k+1$ then its \emph{projective dimension}\/ is $k$.
\emph{Points}, \emph{lines}, \emph{planes}, and \emph{solids}\/ are the
subspaces of $U$ with vector dimension one, two, three, and four, respectively.
For any subspace $Z\subset U$ we let $\cP(Z)$ and $\cL(Z)$ be the set of all
points and lines, respectively, that are contained in $Z$. If $Y$ and $Z$ are
subspaces such that $Y\subset Z\subset U$ then $\cL(Y,Z):=\{M\in\cL(Z)\mid
Y\subset M \}$. In particular, if $p\subset U$ is a point then $\cL(p,U)$ is
the \emph{star of lines}\/ with centre $p$. If, furthermore, $Z\subset U$ is a
plane incident with $p$ then $\cL(p,Z)$ is a \emph{pencil of lines}. We now
could formalise the given projective space in terms of points and lines, but
refrain from doing so.
\par
For the rest of the article, it will be assumed that $V$ is a vector space over
$F$ with vector dimension four, and we shall be concerned with the projective
space $\bPV$.
\par
The exterior square $V\wedge V$ has vector dimension six and gives rise to the
projective space $\bPVV$. (All the multilinear algebra we need can be found in
standard textbooks, like \cite[Sect.\,10.4]{kowa+m-95a} or
\cite[Sect.\,6.8]{kost+m-89a} to mention but a few.) Upon choosing any basis
$e_0,e_1,e_2,e_3$ of $V$, the six bivectors $e_{\sigma\tau}=e_\sigma\wedge
e_\tau $, $0 \leq \sigma < \tau \leq 3$ constitute a basis of $V\wedge V$.
Writing vectors in the form $u=\sum_{\sigma=0}^3 u_\sigma e_\sigma$,
$v=\sum_{\tau=0}^3 v_\tau e_\tau$ with $u_\sigma,v_\tau\in F$ gives
\begin{equation}\label{eq:dachkoo}
    u\wedge v =\sum_{\sigma<\tau} (u_\sigma v_\tau-u_\tau v_\sigma) e_{\sigma\tau} .
\end{equation}
The following results can be found, among others, in
\cite[Sect.\,11.4]{blunck+he-05a}, \cite[Sect.\,15.4]{hirsch-85a},
\cite[Ch.\,34]{pick-76a}, and \cite[Ch.\,xv]{semp+k-98a}. The \emph{Pl\"{u}cker
embedding}
\begin{equation}\label{eq:Pluecker}
    \gamma\colon \cLV\to\cPVV\colon  M\mapsto F(u\wedge v)
\end{equation}
assigns to each line $M\in\cLV$ the point $F(u\wedge v)$, where $u,v\in M$ are
arbitrary linearly independent vectors. If $u$ and $v$ are expressed as above
then the six coordinates appearing on the right hand side of \eqref{eq:dachkoo}
are the \emph{Pl\"{u}cker coordinates}\/ of the line $M$. The image
$\cLV^\gamma=:\cQ$ is the well known \emph{Klein quadric}\/ representing the
lines of $\bPV$. It is given, in terms of coordinates $x_{\sigma\tau}\in F$, by
the quadratic form
\begin{equation}\label{eq:Klein}
    \omega\colon  V\wedge V\to F \colon  \sum_{\sigma<\tau}x_{\sigma\tau}e_{\sigma\tau}
    \mapsto x_{01}x_{23}-x_{02}x_{13}+x_{03}x_{12} .
\end{equation}
Polarisation of $\omega$ gives the non{\trenn}degenerate symmetric bilinear
form
\begin{equation}\label{eq:Kleinpolar}
    \lire \colon  (V\wedge V)^2 \to F \colon  (x,y)\mapsto (x+y)^\omega - x^\omega - y^\omega ,
\end{equation}
whose explicit expression in terms of coordinates is immediate from
\eqref{eq:Klein}. Using the exterior algebra $\bigwedge V$, the form $\lire$
can be characterised via
\begin{equation*}
    x\wedge y = \li x,y \re \,
    (e_0\wedge e_1\wedge e_2\wedge e_3)
    \mbox{~~for all~~} x,y\in V\wedge V\/.
\end{equation*}
The bilinear form $\lire$ defines the \emph{polarity of the Klein quadric},
which sends any subspace $X \subset V\wedge V$ to $X^\perp$. One crucial
property of $\perp$ is as follows: Lines $M,N\in\cLV$ have a point in common
precisely when $M^\gamma$ and $N^\gamma$ are \emph{conjugate points}\/ with
respect to $\perp$, \emph{i.\,e.}, $N^\gamma\subset M^{\gamma\perp}$ (or vice
versa).
\par
A \emph{linear complex of lines}\/ of $\bPV$ is a set $\cH\subset\cL(V)$ whose
image $\cH^\gamma$ is a \emph{hyperplane section}\/ of $\cQ$, \emph{i.\,e.},
$\cH^\gamma=\cP(W)\cap \cQ$ for some hyperplane $W\subset V\wedge V$. The
complex is called \emph{special}\/ if $W$ is tangent to the Klein quadric, and
\emph{general}\/ otherwise. A subset of $\cLV$ is a general linear complex if,
and only if, it is the set of null lines of a null polarity. All lines that
meet or are equal to a fixed line $A\in\cLV$ constitute a special linear
complex with \emph{axis} $A$; all special linear complexes of $\bPV$ are of
this form. A \emph{linear congruence of lines}\/ of $\bPV$ corresponds---via
$\gamma$--- to the section of $\cQ$ by a solid, say $T$. Such a congruence is
said to be \emph{elliptic}\/ if its image under $\gamma$ is an elliptic (or:
\emph{ovoidal}\/ \cite[2.1.4]{schroe-95a}) quadric of $\bP(T)$. This will be
the case precisely when the line $T^\perp$ has no point in common with $\cQ$.
The elliptic linear congruences are precisely the \emph{regular spreads}. There
are three more types of linear congruences, but they will not be needed here.
The Klein quadric has two systems of generating planes. For any plane $G$ of
the first (resp.\ second) system there is a unique point $p$ (resp.\ plane $Z$)
in $\bPV$ such that $\cL(p,V)^\gamma = \cP(G)$ \big(resp.\ $\cL(Z)^\gamma =
\cP(G)$\big). Finally, we recall that a set $\cR\subset \cL(V)$ is a
\emph{regulus}\/ if, and only if, there is a plane $E$ of $\bPVV$ such that
$\cR^\gamma=\cP(E)\cap \cQ$ is a non{\trenn}degenerate conic. For a quick
overview and a detailed description of these and other linear sections of the
Klein quadric we refer to the table in \cite[pp.\,29--31]{hirsch-85a}, even
though over an infinite field some modifications may apply.
\par
The \emph{Grassmann space}\/ $\bGV:=(\cLV,\Pi(V))$ is that partial linear space
whose ``point set'' is the set $\cLV$ of lines of $\bPV$ and whose ``line set''
is the set $\Pi(V)$ of all pencils of lines in $\bPV$. See, for example,
\cite[p.\,71]{pank-10a}. For the sake of readability we henceforth shall
address the ``points'' and ``lines'' of $\bGV$ by their original names. A
\emph{geometric hyperplane} $\cH$ of the Grassmann space $\bGV$ is a proper
subset of $\cLV$ such that each pencil of lines either contains a single
element of $\cH$ or is entirely contained in $\cH$. The geometric hyperplanes
of $\bGV$, which are also called \emph{primes}, are precisely the linear
complexes of lines of $\bPV$. Many proofs were given for this result (and its
generalisation to other Grassmann spaces) \cite{bich+m+t+z-91a},
\cite{debr-09a}, \cite{dere-84a} (finite ground field, see also
\cite[Thm.\,15.2.14]{hirsch-85a}), \cite[p.\,179]{havl-81b} (rephrased in
\cite[Prop.\,3]{havl+z-08a}), and \cite{shult-92a}.

\section{Parallelisms}\label{sect:parallel}
\par
Let $(\cP,\cL)$ be a projective space with point set $\cP$ and line set $\cL$.
An equivalence relation ${\parallel}\subset{\cL\x\cL}$ is a
\emph{parallelism}\/ if each point $p\in\cP$ is incident with precisely one
line from each equivalence class. The equivalence classes of $\parallel$ are
also called \emph{parallel classes}. For further information see
\cite{john-03a} and \cite{john-10a}.
\par

One class of parallelisms is based on the following notions: Let $H$ be an
algebra over a field $F$ such that one of the subsequent
conditions\footnote{Below we identify $F$ with $F\cdot 1_H\subset H$ via
$f\equiv f\cdot 1_H$ for all $f\in F$.} is satisfied:
\begin{enumerate}
\renewcommand{\theenumi}{\Alph{enumi}}
\item\label{A} $H$ is a quaternion skew field with centre $F$.
\item\label{B} $H$ is an extension field of $F$ with degree $[H \mathbin{:}
    F]=4$ and such that $h^2\in F$ for all $h\in H$.
\end{enumerate}
In both cases, $H$ is an infinite \emph{quadratic}\/ (or: \emph{kinematic}\/)
$F${\trenn}algebra, \emph{i.\,e.}, $h^2\in F + Fh$ for all $h\in H$.

\begin{rem}
Let us briefly recall a few facts about the $F${\trenn}algebras appearing in
conditions \eqref{A} and \eqref{B}:
\par
Ad \eqref{A}: Any \emph{quaternion skew field} $H$ with centre $F$ arises as
follows \cite[pp.\,46--48]{tits+w-02a}: We start with a separable quadratic
field extension $K/F$ and denote by $\quer\colon K\to K\colon z\mapsto \ol z$
the only non{\trenn}trivial automorphism of $K$ that fixes $F$ elementwise.
Also, we assume that there is an element $b\in F$ satisfying $b\neq z\ol z$ for
all $z\in K$. Then $(K/F,b)$ is the subring of the matrix ring $K^{2\x 2}$
comprising all matrices of the form
\begin{equation}\label{eq:quat}
    \begin{pmatrix}
    z & w \\ b \ol w & \ol z
    \end{pmatrix}
    \mbox{~~with~~}z,w\in K \mbox{~~arbitrary}.
\end{equation}
We identify $z\in K$ with the matrix $\diag( z,\ol z)\in (K/F,b)$ and, finally,
we adopt the notation $H:=(K/F,b)$.
\par
The \emph{conjugation}\/ $\quer \colon H\to H$ is that antiautomorphism of $H$
which takes a quaternion as in \eqref{eq:quat} to
\begin{equation}\label{eq:conjug}
    \begin{pmatrix}
    \ol z & -w \\ -b \ol w &  z
    \end{pmatrix},
\end{equation}
thereby extending the mapping $\quer$ from above, whence the notation is
unambiguous. The multiplicative \emph{norm function} $H\to F\colon h\mapsto
h\ol h =\ol h h $ is a quadratic form on the $F${\trenn}vector space $H$. The
$F${\trenn}linear form $H\to F\colon h\mapsto h+\ol h$ is the \emph{trace
function}. Any $h\in H$ satisfies the quadratic equation $h^2 -(h+\ol h)h +
h\ol h = 0$ with coefficients in $F$. Upon choosing any $i\in K\setminus F$ the
matrices
\begin{equation}\label{eq:basis}
    1= \begin{pmatrix}
    1 & 0 \\ 0 & 1
    \end{pmatrix},\;\;
    i=\begin{pmatrix}
    i & 0 \\ 0  & \ol i
    \end{pmatrix},\;\;
    j:=\begin{pmatrix}
    0 & 1 \\ b  & 0
    \end{pmatrix},\;\;
    k:=\begin{pmatrix}
    0 & i \\ b \ol i & 0
    \end{pmatrix}
\end{equation}
constitute a basis of $H$ over $F$. It is conventional to choose $i$ as
follows:

\begin{enumerate}
\renewcommand{\theenumi}{A\arabic{enumi}}\renewcommand{\labelsep}{0.6ex}%
\item\label{A1} $\Char F\neq 2$: We may assume that $i^2-a=0$ for some
    $a\in F\setminus\{0\}$, whence $\ol i = -i$, $\ol j=-j$, and $\ol
    k=-k$. This gives the formulas
\begin{equation}\label{eq:mult1}
\begin{aligned}
    i^2&= a,         &j^2&= b,         & k^2&=-ab, \\
    ij &= - ji = k,  &jk &= - kj = -bi,& ki &=-ik=-aj.
\end{aligned}
\end{equation}

\item\label{A2} $\Char F = 2$: We may assume that $i^2+i+a=0$ for some
    $a\in F\setminus\{0\}$, whence $\ol i = i+1$, $\ol j=j$, and $\ol k=k$.
    Now we obtain
\begin{equation}\label{eq:mult2}
\begin{aligned}
    i^2&= i+a,           &j^2& = b,             & k^2&= ab, \\
    ij &= k,             &jk & = b+bi,          & ki &= aj,\\
    ji &=j+k,            &kj & = bi,            & ik &=aj+k.
\end{aligned}
\end{equation}
\end{enumerate}
We refer to \cite[p.\,169]{pick-75a} for a different basis of $H$ over $F$. It
has the advantage to be applicable in any characteristic for quaternion skew
fields, but using it would not allow us to incorporate case \eqref{B} in the
way we do below. See also \cite{maeu-99a} for a characterisation of arbitrary
quaternion skew fields. An analogous characterisation for real quaternions may
be found in \cite[p.\,43]{graeu-62a}.
\par
Ad \eqref{B}: Here $\Char F=2$, since there exists an $h\in H\setminus F$ for
which $(1+h)^2=1+2h+h^2\in F$ implies $2h\in F$, and so $2=0$. The field $H$ is
a purely inseparable extension of $F$, and it fits formally into the
description from \eqref{A1} if we proceed as follows: First, we select
arbitrary elements $i,j\in H$ such that $1,i,j$ are linearly independent over
$F$. Next, we let $k:=ij$, $a:=i^2\in F\setminus\{0\}$, $b:=j^2\in
F\setminus\{0\}$, and we regard the identity mapping as being the
\emph{conjugation}\/ $\quer\colon H\to H$. Then, taking into account that minus
signs can be ignored due to $\Char F=2$, the multiplication in $H$ is given by
the formulas in \eqref{eq:mult1}. (One may also carry over formulas
\eqref{eq:quat} \eqref{eq:conjug}, and \eqref{eq:basis} by letting
$K:=F(i)\subset H$ and $\ol z:=z$ for all $z\in K$.) The \emph{norm}\/ and
\emph{trace}\/ of $h\in H$ are defined as for quaternions. So the norm of $h\in
H$ is $h^2$ and its trace is $h+h=0$. The polar form of the quadratic norm form
$H\to F$ is the zero bilinear form.
\end{rem}

We now consider the (three{\trenn}dimensional) projective space $\bPH$ on the
$F$-vector space $H$. Our next definition, which follows
\cite[pp.\,112--115]{john-03a} and \cite[\S\,14]{karz+k-88a}, makes use of the
multiplicative group $H\setminus\{0\}$.

\begin{defn}\label{def:Clifford}
Given a pair $(M,N)$ of lines of the projective space $\bPH$ we say that $M$ is
\emph{left parallel}\/ to $N$, in symbols $M\lip N$, if there is an element
$c\in H\setminus\{0\}$ such that $cM=N$. Similarly, $M$ is said to be
\emph{right parallel}\/ to $N$, in symbols $M\rep N$, if $Mc=N$ for some $c\in
H\setminus\{0\}$.
\end{defn}
The relations $\lip$ and $\rep$ are parallelisms, which are identical precisely
in case~\eqref{B} \cite[\S~14]{karz+k-88a}. The left and right parallel class
of any line $M$ of $\bPH$ will be written as $\liS(M)$ and $\reS(M)$,
respectively.
\begin{defn}\label{def:Cliff-canon}
The parallelisms $\lip$ and $\rep$ are called the \emph{canonical Clifford
parallelisms}\/ of $\bPH$.
\end{defn}

Finally, we extend the previous definition to a projective space $\bPV$ as in
Section~\ref{sect:prelim}:

\begin{defn}\label{def:Cliff-allg}
A parallelism $\parallel$ of a three{\trenn}dimensional projective space $\bPV$
is said to be \emph{Clifford}\/ if the $F$-vector space $V$ can be made into an
$F${\trenn}algebra $H:=V$ subject to \eqref{A} or \eqref{B} such that the given
parallelism $\parallel$ coincides with one of the canonical Clifford
parallelisms of $\bPH$.
\end{defn}

\section{External planes to the Klein quadric}\label{sect:external}

In this section we adopt the settings from Section~\ref{sect:prelim}. Our
starting point is a very simple one, namely that of a plane $C$ in $\bPVV$
external to the Klein quadric. In other words, $C$ has to satisfy the following
property:
\begin{equation}\label{eq:ext}
    C \mbox{~~has no point in common with the Klein quadric~~}\cQ.
\end{equation}
The restriction to $C$ of the quadratic form $\omega$ from \eqref{eq:Klein}
defines a quadric without points in the projective plane $\bP(C)$.
Consequently, a plane of this kind cannot exist over certain fields, like
quadratically closed fields or finite fields; see \cite[p.\,4]{hirsch+th-91a}.
The following simple lemma will be used repeatedly.
\begin{lem}\label{lem:schnitt}
Let $G$ be a plane that lies entirely on the Klein quadric $\cQ$, and let
$T\supset C$ be a subspace of\/ $\bPVV$ with projective dimension $k$. Then
$G\cap T$ has projective dimension $k-3$.
\end{lem}
\begin{proof}
From \eqref{eq:ext} we obtain $G\cap C=0$, whence $V\wedge V=G\oplus C$ by the
dimension formula. Consequently, $V\wedge V=G+T$ and applying again the
dimension formula proves the assertion.
\end{proof}

We now use the given plane $C$ and the Pl\"{u}cker embedding $\gamma$ from
\eqref{eq:Pluecker} to define a parallelism of the projective space $\bPV$.

\begin{defn}\label{def:C-par}
Given any pair $(M,N)$ of lines in the projective space $\bPV$ we say that $M$
is \emph{$C${\trenn}parallel}\/ to $N$, in symbols $M\Cp N$, if $ C+M^\gamma =
C+N^\gamma$. In addition, we define $\CS(M):=\{X\in \cL(V) \mid M \Cp X \}$.
\end{defn}

\begin{prop}\label{prop:par}
The relation\/ $\Cp$ is a parallelism of\/ $\bPV$. All its parallel classes are
regular spreads.
\end{prop}
\begin{proof}
Obviously, $\Cp$ is an equivalence relation on $\cL(V)$. Given any line $M$ and
any point $p$ in $\bPV$ we consider the star $\cL(p,V)$. The image
$\cL(p,V)^\gamma$ is the point set of a plane, say $G$, lying entirely on
$\cQ$. By Lemma~\ref{lem:schnitt}, $G\cap (C+M^\gamma)$ is a single point,
whose preimage under $\gamma$ is the only line through $p$ that is
$C${\trenn}parallel to $M$.

\par
Consider the parallel class $\CS(M)$ of any line $M\in\cL(V)$. Then $\CS(M)$ is
a spread of $\bPV$. By definition, the image $\CS(M)^\gamma$ is that quadric in
the solid $C+M^\gamma$ which arises as section of the Klein quadric by
$C+M^\gamma$. Since $\CS(M)$ is a spread, the quadric $\CS(M)^\gamma$ contains
more than one point. From \eqref{eq:ext} there cannot be a line on
$\CS(M)^\gamma$. So the quadric $\CS(M)^\gamma$ is elliptic. This shows that
$\CS(M)$ is a regular spread.
\end{proof}

See \cite[Lemma\,9]{bett+r-05a} and \cite[Def.\,1.10]{bett+r-12a} for a version
of Proposition~\ref{prop:par} in the classical context. Planes that are
external to the Klein quadric arise naturally in elliptic line geometry (over
the real numbers) \cite[pp.\,339--342]{weiss-82a}. Proposition~\ref{prop:par}
appears also in \cite[Sect.\,3]{havl-95c} in the setting of \emph{generalised
elliptic spaces}. However, our current approach shows that the elliptic
polarity used there is superfluous when exhibiting a Clifford parallelism on
its own. Even more, avoiding such an elliptic polarity in $\bPV$ allows us to
treat the subject in full generality, whereas \cite{havl-95c} will not tell us
anything about the first case in
Proposition~\ref{prop:Cperp}~\eqref{prop:Cperp3} below.

\begin{prop}\label{prop:bijektion}
For each parallel class\/ $\CS$ of\/ $\Cp$ there is a unique solid, say $T$,
in\/ $\bPVV$ such that the section of the Klein quadric by $T$ equals
$\CS^\gamma$. Furthermore, the plane $C$ is contained in any such $T$.
Conversely, any solid in\/ $\bPVV$ that contains the plane $C$ arises in this
way from precisely one parallel class of\/ $\Cp$.
\end{prop}

\begin{proof}
By choosing some $M\in\CS$, we obtain from Definition~\ref{def:C-par} that
$\CS^\gamma$ is the section of the Klein quadric by the solid $T:=C+M^\gamma$.
Next, we read off from the proof of Proposition~\ref{prop:par} that
$\CS^\gamma$ is an elliptic quadric, whence its span is a solid, which clearly
coincides with $T$. So our $T$ is uniquely determined by $\CS$ and, by its
definition, contains the plane $C$.
\par
If $T\supset C$ is a solid then, upon choosing a plane $G$ on the Klein quadric
$\cQ$ and by applying Lemma~\ref{lem:schnitt}, we see that there is a line
$M\in\cL(V)$ with $M^\gamma=T\cap G$. So, by the above, $\CS(M)^\gamma$
generates the solid $T$. Since $\CS(M)^\gamma$ \emph{is}\/ the section of $\cQ$
by the solid $T$ (and not only a subset of this section), no parallel class
other than $\CS(M)$ gives rise to $T$.
\end{proof}

\begin{cor}\label{cor:recover}
The plane $C$ in\/ $\bPVV$ can be uniquely recovered from any two distinct
parallel classes of the parallelism\/ $\Cp$ of\/ $\bPV$.
\end{cor}

\begin{prop}\label{prop:Cperp}
Let $C^\perp$ be the polar plane of $C$ with respect to the Klein quadric\/
$\cQ$. Then the following assertions hold:
\begin{enumerate}
\item\label{prop:Cperp1} The plane $C^\perp$ is external to\/ $\cQ$.

\item\label{prop:Cperp2} If\/ $\Char F\neq 2$ then the planes $C$ and
    $C^\perp$ have no point in common.

\item\label{prop:Cperp3} If\/ $\Char F=2$ then either $C\cap C^\perp$ is a
    single point or $C=C^\perp$.
\end{enumerate}
\end{prop}
\begin{proof}
Ad \eqref{prop:Cperp1}: Let $q\subset C^\perp$ be a point. Then $q^\perp\supset
C$ is a hyperplane in $\bPVV$. By Lemma~\ref{lem:schnitt}, $G\cap q^\perp$ is a
line for all planes $G$ on the Klein quadric $\cQ$. Any tangent hyperplane of
$\cQ$ contains a plane that lies entirely in $\cQ$. Thus $q^\perp$ cannot be
tangent to $\cQ$. This in turn shows that $q\notin\cQ$.

Ad \eqref{prop:Cperp2}: Due to $\Char F\neq 2$, an arbitrary point $p $ of
$\bPVV$ belongs to $\cQ$ if, and only if, $p \subset p ^\perp$. From this
observation and \eqref{eq:ext}, we obtain $p \not\subset p ^\perp\supset
C^\perp$ for all points $p \subset C$.

Ad \eqref{prop:Cperp3}: Now $\Char F=2$ forces the bilinear form $\lire$ from
\eqref{eq:Kleinpolar} to be symplectic, whence $\perp$ is a null polarity. The
restriction of $\lire$ to $C\x C$ is an alternating bilinear form with radical
$C\cap C^\perp$. So the vector dimension of the quotient vector space $C
/(C\cap C^\perp)$ has to be even. This implies that $C\cap C^\perp$ is either a
point or a plane. In the latter case we clearly have $C=C^\perp$.
\end{proof}

\par

By Propositions~\ref{prop:par} and \ref{prop:Cperp}, the plane $C^\perp$ also
gives rise to a parallelism, which will be denoted by $\Cpp$. Clearly, the role
of $C$ and $C^\perp$ is interchangeable.

\begin{prop}\label{prop:opp}
Let\/ $\cR\subset \cL(V)$ be a regulus whose lines are mutually
$C${\trenn}parallel. Then the lines of its opposite regulus $\cR'$ are mutually
$C^\perp${\trenn}parallel.
\end{prop}
\begin{proof}
The image $\cR^\gamma$ is the section of $\cQ$ by a (uniquely determined)
plane, say $E$. Then ${\cR'}^{\,\gamma}$ is the section of $\cQ$ by the plane
$E^\perp$. Let $\CS$ denote the $C${\trenn}parallel class that contains $\cR$.
By Proposition~\ref{prop:bijektion}, its image $\CS^\gamma$ is the section of
$\cQ$ by a uniquely determined solid, say $T$, with $T\supset C$. Now $T=E+C$
implies that $E\cap C$ is a line, whence ${E^\perp}+ {C^\perp}$ is a solid.
Thus the lines of $\cR'$ are mutually $C^\perp${\trenn}parallel.
\end{proof}

We are now in a position to show our first main result, namely that $\bP(V)$
together with our parallelisms $\Cp$ and $\Cpp$ is a \emph{double space}\/
\cite[p.\,113]{john-03a}, \cite[p.\,75]{karz+k-88a}. This amounts to verifying
the \emph{double space axiom}, which in our setting reads as follows:

\begin{itemize}
\item[(D)] For any three non{\trenn}collinear points $p$, $q$, $r$ in
    $\bP(V)$ the unique line through $r$ that is $C${\trenn}parallel to
    $p\oplus q$ has a point in common with the unique line through $q$ that
    is $C^\perp${\trenn}parallel to $p\oplus r$.
\end{itemize}

\begin{thm}\label{thm:doublespace}
The parallelisms\/ $\Cp$ and\/ $\Cpp$ turn the projective space\/ $\bPV$ into a
double space. This implies that the $F${\trenn}vector space $V$ can be made
into an $F${\trenn}algebra $H:=V$ subject to\/ \emph{\eqref{A}} or\/
{\eqref{B}} such that the canonical Clifford parallelisms\/ $\lip$ and\/ $\rep$
of\/ $\bPH$ coincide with\/ $\Cp$ and\/ $\Cpp$, respectively.
\end{thm}
\begin{proof}
With the notation from (D), let $M:=p\oplus q$ and $N:=p\oplus r$. The set of
all lines from $\CS(M)$ that meet $N$ in some point is a regulus $\cR$, say.
One line of $\cR$ is the unique line $M_1$ satisfying $r\subset M_1\Cp M$. The
point $q$ is incident with a unique line $N_1$ of the regulus $\cR'$ opposite
to $\cR$. Hence $M_1$ and $N_1$ have a point in common. From $N\in\cR'$ and
Proposition~\ref{prop:opp}, this $N_1$ is at the same time the only line
through $q$ that satisfies $N_1\Cpp N$.

By \cite[(14.2) and (14.4)]{karz+k-88a} or \cite[Thms.\,1--4]{john-03a}, where
the work of numerous authors is put together, the vector space $V$ can be
endowed with a multiplication that makes it into an $F${\trenn}algebra with the
required properties.
\end{proof}

\section{From Clifford towards Klein}\label{sect:Cliff-Klein}

Let us turn to the problem of reversing Theorem~\ref{thm:doublespace}. We
assume that an $F${\trenn}algebra $H$ is given according to condition \eqref{A}
or \eqref{B} from Section~\ref{sect:parallel}. We aim at describing the
canonical Clifford parallelisms of $\bPH$ in terms of the ambient space $\bPHH$
of the Klein quadric $\cQ$. (Notations that were introduced for $V\wedge V$
will be used \emph{mutatis mutandis}\/ also for $H\wedge H$.) To this end let
$H/F$ be the \emph{quotient vector space}\/\footnote{The symbol $H/F$ will
exclusively be used to denote this quotient space rather than to express that
$H$ is a skew field extension of $F$.} of the $F${\trenn}vector space $H$
modulo its subspace $F$. The mapping
\begin{equation}\label{eq:beta}
      \beta\colon    H\x H \to H/F \colon  (g,h)\mapsto \ol g h + F
\end{equation}
is $F${\trenn}bilinear and alternating, since for all $h\in H$ the norm $\ol h
h$ is in $0+F\in\ H/F$. By the universal property of the exterior square
$H\wedge H$, there is a unique $F${\trenn}linear mapping
\begin{equation}\label{eq:kappa}
    \kappa \colon  H\wedge H \to H/F \mbox{~~such that~~} (g\wedge h)^\kappa = (g,h)^\beta
    \mbox{~~for all~~}g,h\in H,
\end{equation}
and we define
\begin{equation}\label{eq:C=ker}
    C:=\ker \kappa .
\end{equation}
Our $\kappa$ is surjective, due to $h+F=(1\wedge h)^\kappa$ for all $h \in H$.
So the kernel of $\kappa$ has vector dimension $6-3$, \emph{i.\,e.},
$C=\ker\kappa$ is a plane in $\bPHH$. In analogy to \eqref{eq:beta} and
\eqref{eq:kappa}, the alternating $F${\trenn}bilinear mapping
\begin{equation}\label{eq:beta'}
    \beta'\colon   H\x H \to H/F \colon  (g,h)\mapsto g\ol h + F
\end{equation}
gives rise to a uniquely determined surjective $F${\trenn}linear mapping
$\kappa'\colon H\wedge H\to H/F$ such that $(g\wedge h)^{\kappa'} = (g ,
h)^{\beta'}$ for all $g,h\in H$. Therefore, a second plane in $\bPHH$ is given
by
\begin{equation}\label{eq:C'=ker}
    C': = \ker\kappa' .
\end{equation}

\begin{thm}\label{thm:C+C'}
In\/ $\bPHH$ the plane $C=\ker\kappa$ is external to the Klein quadric $\cQ$.
The canonical Clifford parallelism\/ $\lip$ of\/ $\bPH$ coincides with the
parallelism\/ $\Cp$ that arises from the plane $C$ according to
Definition\/~\emph{\ref{def:C-par}}. A similar result holds for the plane
$C'=\ker\kappa'$, the canonical Clifford parallelism\/ $\rep$, and the
parallelism\/ $\Ccp$. Furthermore, $C$ and $C'$ are mutually polar under the
polarity of the Klein quadric.
\end{thm}
\begin{proof}
We choose any point of $\cQ$; it can be written in the form $F(g\wedge h)$ for
some elements $g,h\in H$ that are linearly independent over $F$. We have
\begin{equation*}
    g\wedge x \in C \Leftrightarrow \ol g x \in F
   \Leftrightarrow g^{-1} x \in F
    \Leftrightarrow x \in gF = F g
    \mbox{~~for all~~}x\in H.
\end{equation*}
So the arbitrarily chosen point $F(g\wedge h)$ is not incident with $C$.
\par
The multiplicative group $H\setminus\{0\}$ acts on $H$ via left multiplication.
More precisely, for any $c\in H\setminus\{0\}$ we obtain the \emph{left
translation} $\lambda_c\colon H\to H\colon h\mapsto c h$. As $\lambda_c$ is an
$F${\trenn}linear bijection, so is its exterior square ${\lambda_c \ulwedge
\lambda_c}\colon H\wedge H\to H\wedge H$. This exterior square describes the
action of $\lambda_c$ on the line set $\cLH$ in terms of bivectors, as it takes
any pure bivector $g\wedge h\in H\wedge H$ to $c g\wedge c h$. So we may read
off from $\ol c c = c\ol c\in F\setminus\{0\}$ and
\begin{equation*}\label{}
    (c g\wedge c h)^\kappa = \ol{c g}c h + F = \ol g (\ol c c) h + F  = (c\ol c) (g\wedge h)^\kappa
\end{equation*}
that
\begin{equation}\label{eq:EV}
    C + F(c g \wedge c h ) = C + F(g \wedge h) \mbox{~~for all~~}g,h\in H.
\end{equation}
As the pure bivectors span $H\wedge H$, formula \eqref{eq:EV} implies that all
subspaces of $H\wedge H$ passing through $C$ are invariant under
$\lambda_c\ulwedge\lambda_c$.
\par
Now let $M\lip N$, whence there is a particular $c\in H\setminus\{0\}$ with
$cM=N$. By the above, the subspace $C+M^\gamma$ is invariant under
$\lambda_c\ulwedge\lambda_c$, so that $C+M^\gamma = C+ (cM)^\gamma$ or, said
differently, $M\Cp N$. We obtain as an intermediate result that every left
parallel class is a subset of a $C${\trenn}parallel class. However, a left
parallel class cannot be properly contained in a $C${\trenn}parallel class, for
then it would not cover the entire point set of $\bP(H)$.
\par
By switching from left to right and replacing $\kappa$ with $\kappa'$, the
above reasoning shows that the right parallelism $\rep$ coincides with $\Ccp$.
\par
Finally, we establish that the polarity $\perp$ of the Klein quadric takes the
plane $C$ to the plane $C'$. Here we use the well known result that each of the
two (not necessarily distinct) parallelisms of a projective double space
determines uniquely the other parallelism \cite[p.\,76]{karz+k-88a}. So, by the
above and Theorem~\ref{thm:doublespace}, we obtain from ${\lip}={\Cp}$ that
${\Ccp}={\rep}={\Cpp}$. Now Corollary~\ref{cor:recover} shows $C'=C^\perp$.
\end{proof}
\par
By virtue of Theorem~\ref{thm:doublespace} and the preceding theorem, we have
established the announced one-one correspondence between Clifford parallelisms
and planes that are external to the Klein quadric.

\begin{rem}
As our approach to the planes $C$ and $C'=C^\perp$ in \eqref{eq:C=ker} and
\eqref{eq:C'=ker} is somewhat implicit, it seems worthwhile to write down a
basis for each of these planes. This can be done as follows: First we apply
$\kappa$ and $\kappa'$ to the six basis elements $1\wedge i, 1\wedge j,\ldots,
j\wedge k$ of $H\wedge H$, which then allows us to find three linearly
independent bivectors in $\ker\kappa$ and $\ker\kappa'$, respectively. See
Table~\ref{tab:basen} and take notice that in case~\eqref{A2} the point $C\cap
C^\perp$ is given by the bivector $b\wedge i + j\wedge k$.
\begin{table}[!h]\renewcommand{\arraystretch}{1.2}
\begin{equation*}
    \begin{array}{|c||c||c|r|r|}
    \hline
    \mbox{Cases}&\mbox{Plane}&\multicolumn{3}{c|}{ \mbox{Basis}}\\
    \hline
    \!\eqref{A1},\,  \eqref{B}\!& C\phantom{^\perp} &
     b\wedge i - j\wedge k  &  a \wedge j + i\wedge k& 1\wedge k + i \wedge j   \\
    \cline{2-5}
    &C^\perp&
    b\wedge i + j\wedge k  &  a\wedge j- i\wedge k  &  1\wedge k-i\wedge j  \\
    \hline
    \eqref{A2}& C\phantom{^\perp} &
    b\wedge i + j\wedge k  & a\wedge j + i\wedge k  &  1\wedge(k+j) +i\wedge j \\
    \cline{2-5}
    & C^\perp &
    b\wedge i+j\wedge k &  1\wedge (aj+k) + i\wedge k  & 1\wedge k+ i\wedge j  \\
    \hline
    \end{array}
\end{equation*}
\caption{Bases of the planes $C$ and $C^\perp$.}\label{tab:basen}
\end{table}
\par
Alternatively, we may consider hyperplanes of $\bPHH$ that are incident with
$C$. Any such hyperplane can be obtained as the kernel of an $F${\trenn}linear
form $H\wedge H\to F$ as follows: We choose any non{\trenn}zero
$F${\trenn}linear form $\phi\colon H\to F$ such that $1^\phi=0$. Then the
mapping
\begin{equation}\label{eq:altblf}
    H\x H\to F \colon  (g,h)\mapsto (\ol g h)^\phi
\end{equation}
is an alternating $F${\trenn}bilinear form. The universal property of $H\wedge
H$ gives the existence of a unique $F${\trenn}linear form $ \psi \colon H\wedge
H\to F$ such that $(g\wedge h)^\psi = (\ol g h)^\phi$ for all $g,h\in H$, and
clearly $C\subset \ker\psi$. In this way $C$ can be described as intersection
of three appropriate hyperplanes. By replacing $\ol g h$ with $g \ol h$, a
similar result is obtained for $C^\perp$. In case~\eqref{A2}, \emph{i.\,e.}, if
$H$ is a quaternion skew field and $\Char F=2$, the trace form is a
distinguished choice of $\phi$. It turns \eqref{eq:altblf} into the polar form
of the norm: $(g,h)\mapsto \ol g h + \ol h g $. Furthermore, the hyperplane
arising from the trace form is equal to $C+C^\perp$, due to $\ol g h + \ol h g
= g \ol h + h \ol g$ for all $g,h\in H$.
\end{rem}

\begin{rem}
The mappings \eqref{eq:beta} and \eqref{eq:beta'} admit a geometric
interpretation by considering the projective plane $\bP(H/F)$. The ``points''
of this projective plane can be identified with the lines of the star
$\cL(F1,H)$ via $F(h+F)\mapsto F1\oplus Fh$ for all $h\in H\setminus F$. We
define a mapping $\cLH\to\cL(F1,H)$ by assigning to each line $M$ the only line
of the parallel class $\liS(M)$ through the point $F1$. Letting $M=Fg\oplus Fh$
with (linearly independent) $g,h\in H$, we obtain $F1\oplus F(g^{-1}h) = F(\ol
g h) \oplus F1$ as image of $M$. So the vector $\ol gh + F\in H/F$ appearing in
\eqref{eq:beta} is a representative of the image of $M$. The interpretation of
\eqref{eq:beta'} in terms of $\rep$ is similar. In the classical setting such a
mapping is known under its German name \emph{Eckhart-Rehbock Abbildung}; see
\cite{brau-73a}, \cite{wund-60a}, and the references given there.
Generalisations, in particular to Lie groups, are the topic of
\cite[pp.\,16--17]{brau-73a}, \cite{husty+n-93}, and \cite{lueb-79a}.
\par
Alternating mappings like the ones from \eqref{eq:beta} and \eqref{eq:beta'}
appear in the definition of \emph{generalised Heisenberg algebras}; see, for
example, \cite[Def.\,6.1]{stro-08a}. These algebras are important in the
classification of certain \emph{nilpotent Lie algebras}. We encourage the
reader to take a closer look at \cite{cica+d+s-12a},
\cite[Sect.\,7]{gulde+st-12a}, and \cite[Sect.\,8]{stro-08a}, in order to see
how planes external to the Klein quadric have successfully been utilised in
that context.
\end{rem}

\begin{rem}
Our proof of Theorem~\ref{thm:C+C'} has shown the following: All hyperplanes of
$\bPHH$ that are incident with $C$ are invariant under the exterior square of
any left translation $\lambda_c$. Since ${\lambda_c \ulwedge \lambda_c}$
commutes with the polarity of the Klein quadric, we immediately obtain that
${\lambda_c \ulwedge \lambda_c}$ fixes all points of the plane $C^\perp$. A
similar result holds for the exterior square of any \emph{right translation}
$\rho_c\colon H\to H \colon h\mapsto h c$. However, we do not enter into a
detailed discussion of these mappings. In this regard, the articles
\cite{graeu-62a} and \cite{zadd-89a} about real quaternions deserve special
mention.
\end{rem}

\section{Characterisations of a single Clifford parallelism}\label{sect:character}

In this section we consider again a three{\trenn}dimensional projective space
$\bPV$ as described in Section~\ref{sect:prelim}.

\begin{lem}\label{lem:special}
No spread of\/ $\bPV$ is contained in a special linear complex of lines.
\end{lem}
\begin{proof}
Assume to the contrary that a spread $\cS$ is contained in a special linear
complex with axis $A\in\cLV$, say. Choose any point $p_1$ off the line $A$ and
let $M_1\in\cS$ be the line through $p_1$. By our assumption, $A+M_1$ is a
plane, and in this plane there is a point $p_2$ that lies neither on $A$ nor on
$M_1$. The plane $A+M_1$ contains also the line $M_2\in\cS\setminus\{M_1\}$
through $p_2$, whence $M_1$ and $M_2$ have a unique common point, an absurdity.
\end{proof}
We add in passing that Lemma~\ref{lem:special} is closely related with a result
\cite[Prop.\,6.10\,(4)]{debr-11a} about a specific class of geometric
hyperplanes arising from regular spreads of lines (for arbitrary odd projective
dimension). We now show our first characterisation of a Clifford parallelism:

\begin{thm}\label{thm:hyper}
For any parallelism\/ $\parallel$ of\/ $\bPV$ the following properties are
equivalent:
\begin{enumerate}
\item\label{thm:hyper1} The parallelism is Clifford.

\item\label{thm:hyper2} Any two distinct parallel classes are contained in
    at least one geometric hyperplane of the Grassmann space\/ $\bGV$.

\item\label{thm:hyper3} Any two distinct parallel classes are contained in
    a unique linear complex of\/ $\bPV$. This linear complex is general.
\end{enumerate}
\end{thm}
\begin{proof}
\eqref{thm:hyper1}~$\Rightarrow$~\eqref{thm:hyper2}: By
Definition~\ref{def:Cliff-allg}, we can apply Theorem~\ref{thm:C+C'}. This
shows that ${\parallel}$ is one of the parallelisms ${\Cp}$ and $\Ccp$ from
there. So, up to a change of notation, we may assume ${\parallel}={\Cp}$.
Suppose that $\cS_1\neq\cS_2$ are parallel classes of ${\parallel}$. The image
$\cS_1^\gamma$ is the section of the Klein quadric by a solid $T_1\supset C$
according to Definition~\ref{def:C-par}. Likewise, all points of $\cS_2^\gamma$
are contained in a solid $T_2\supset C$. Since $T_1$ and $T_2$ have the plane
$C$ in common, there exists a hyperplane $W$ of $\bPVV$ such that $W\supset
T_1+T_2$. This $W$ is incident with all points of
$\cS_1^\gamma\cup\cS_2^\gamma$, and so $ \{X\in\cLV \mid X^\gamma \subset W\}$
is a linear complex of lines (or, said differently, a geometric hyperplane)
containing $\cS_1\cup \cS_2$.

\par
\eqref{thm:hyper2}~$\Rightarrow$~\eqref{thm:hyper3}: Suppose that parallel
classes $\cS_1\neq\cS_2$ are both contained in at least one geometric
hyperplane. Let $\cH$ be any of these. We have noticed in
Section~\ref{sect:prelim} that $\cH$ is a linear complex of lines which, by
Lemma~\ref{lem:special} applied to $\cS_1$, has to be a general.
\par
Now let $\cH'$ and $\cH''$ be general linear complexes of lines both containing
$\cS_1\cup\cS_2$. Choose any point $p$ of $\bPV$. Through $p$ there are
uniquely determined lines $M_1\in\cS_1$ and $M_2\in\cS_2$. Also we have
$M_1\neq M_2$. Each of the intersections $\cH'\cap\cL(p,V)$ and
$\cH''\cap\cL(p,V)$ is a pencil of lines. Both pencils have to contain the
lines $M_1$ and $M_2$, and therefore these pencils are identical. As $p$ varies
in the point set of $\bPV$, this gives $\cH'=\cH''$.

\par
\eqref{thm:hyper3}~$\Rightarrow$~\eqref{thm:hyper1}: Choose any parallel class
$\cS(M)$ with $M\in \cLV$. Our first aim is to show that the points of
$\cS(M)^\gamma$ generate a solid in $\bPVV$. Let $q$ be a point of $M$. Through
$q$ there are lines $M_1$ and $M_2$ such that $M$, $M_1$, and $M_2$ are not
coplanar. Denote by $\cH(M,M_1)$ and $\cH(M,M_2)$ the uniquely determined
general linear complexes that contain $\cS(M)\cup\cS(M_1)$ and
$\cS(M)\cup\cS(M_2)$, respectively. Also let $W_1$ and $W_2$ be the hyperplanes
of $\bPVV$ corresponding to these linear complexes. Then
$\cH(M,M_1)\cap\cL(q,V)$ is a pencil of lines, which contains $M$ and $M_1$,
but not $M_2$. Thus $\cH(M,M_1)\neq \cH(M,M_2)$. This gives $W_1\neq W_2$, and
therefore $T:=W_1\cap W_2$ turns out to be a solid of $\bPVV$. The
$\gamma$-preimage of the point set $\cP(T)$ is a linear congruence of lines of
$\bPV$, say $\cE$, which contains the spread $\cS(M)$ as a subset. No
hyperplane of $\bPVV$ through $T$ can be tangent to the Klein quadric by
Lemma~\ref{lem:special}, whence the line $T^\perp$ is exterior to the Klein
quadric. This means that the linear congruence $\cE$ is elliptic and hence a
regular spread. As the spread $\cS(M)$ is contained in the spread $\cE$, these
two spreads are identical. Therefore $\cS(M)^\gamma$, due to its being an
elliptic quadric, generates the solid $T$.
\par
According to the previous paragraph, we may assign to each parallel class $\cS$
a uniquely determined solid $T$ of $\bPVV$. Let $\cF$ be the set of all such
solids. The assignment $\cS\mapsto T$ is injective, since $\cS^\gamma$ is the
section of the Klein quadric by $T$. Hence $\cF$ comprises more than one solid.
From our assumption in \eqref{thm:hyper3}, any two distinct solids from $\cF$
are contained in a unique hyperplane of $\bPVV$. This implies (see, for
example, \cite[Prop.\,3.2]{pank-10a}) that at least one of the following
assertions holds:
\begin{itemize}
\item[(i)] There is a unique hyperplane of $\bPVV$ that is incident with
    all solids belonging to $\cF$.
\item[(ii)] There is a unique plane of $\bPVV$ that is incident with all
    solids belonging to $\cF$.
\end{itemize}
The situation from (i) cannot occur in our setting, since it would imply that
all parallel classes, and hence all of $\cLV$, would belong to a single linear
complex of lines. Consequently, only (ii) applies, and we let $C$ be this
uniquely determined plane. We establish that this $C$ is external to the Klein
quadric: Indeed, any common point would be the $\gamma$-image of a line
belonging to \emph{all}\/ parallel classes, which is utterly absurd.
Consequently, ${\parallel}$ coincides with ${\Cp}$, and
Theorem~\ref{thm:doublespace} shows that $\parallel$ is Clifford.
\end{proof}

For a version of the previous result, limited to the real case and with
somewhat different assumptions, we refer to \cite[Lemma\,14]{bett+r-05a} and
\cite[Def.\,1.9]{bett+r-12a}. Our second characterisation is based on the
following \emph{condition of crossed pencils}\/ (Figure~\ref{abb:crossed}) for
an arbitrary parallelism $\parallel$ of $\bPV$:
\begin{itemize}
\renewcommand{\labelsep}{0.3ex}
\item[(CP)] For all lines $M_1\parallel N_1$ and $M_2\parallel N_2$ such
    that $p:=M_1\cap M_2$ and $q:=N_1\cap N_2$ are two distinct points the
    following holds: If the lines $M_1$, $M_2$, and $p\oplus q$ are in a
    common pencil, so are the lines $N_1$, $N_2$, and $p\oplus q$.
\end{itemize}
\begin{figure}[!ht]\unitlength1cm
  \centering
  \begin{picture}(4.5,3)
  \small
    \put(0,0.0){\includegraphics[height=8\unitlength]{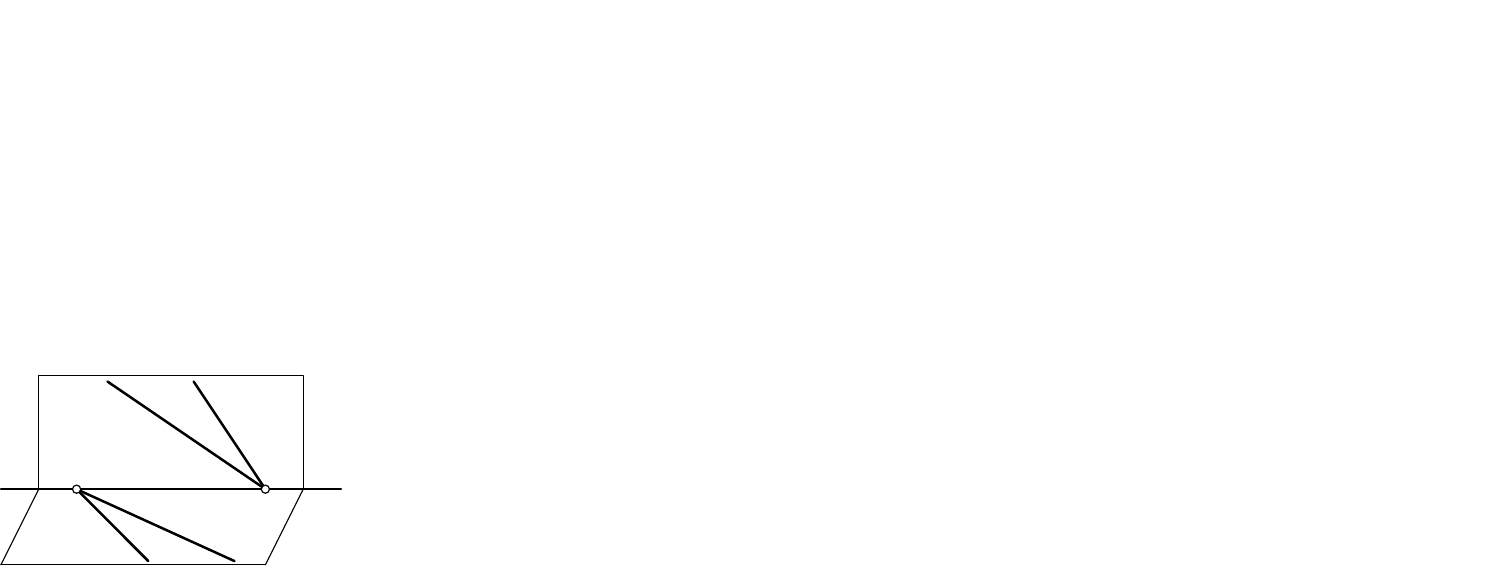}}
    \put(0.9,0.8){$p$}
    \put(3.7,1.3){$q$}
    \put(1.3,0.15){$M_1$}
    \put(3.25,0.15){$M_2$}
    \put(1.2,2.3){$N_1$}
    \put(3.0,2.3){$N_2$}
  \end{picture}
   \caption{Condition of crossed pencils}\label{abb:crossed}
\end{figure}
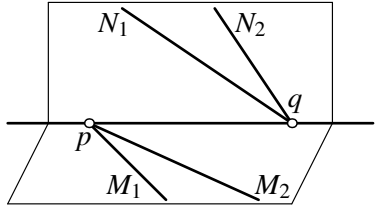

\begin{thm}\label{thm:crossed}
A parallelism\/ $\parallel$ of\/ $\bPV$ is Clifford if, and only if, it
satisfies the condition of crossed pencils.
\end{thm}
\begin{proof}
Let $\parallel$ be Clifford. We consider lines $M_1$, $M_2$, $N_1$, $N_2$
subject to the assumptions in (CP). By Theorem~\ref{thm:hyper}, there is a
unique general linear complex of lines containing $\cS(M_1)\cup \cS(M_2)$. This
complex is the set of null lines of a null polarity $\pi$, say. So
$p^\pi=(M_1\cap M_2)^\pi=M_1+M_2$ and $q^\pi=(N_1\cap N_2)^\pi=N_1+N_2$. Now if
$M_1$, $M_2$, $p\oplus q$ are in a common pencil of lines then $q\subset p^\pi$
implies $p\subset q^\pi$, and this shows that $N_1$, $N_2$, and $p\oplus q$
belong to the pencil $\cL(q,q^\pi)$.
\par
For a proof of the converse we consider any two distinct parallel classes
$\cS_1$ and $\cS_2$. For each point $p$ in $\bPV$ we define $M_1(p)$ as the
only line satisfying $p\subset M_1(p)\in\cS_1$; the line $M_2(p)$ is defined
analogously. We obtain a well defined mapping $\pi$ of the point set $\cPV$
into the set of planes of $\bPV$ via
\begin{equation*}\label{}
    p \mapsto p^\pi:= M_1(p)+M_2(p) .
\end{equation*}
We claim that $q\subset p^\pi$ implies $p\subset\ q^\pi$ for all $p,q\in \cPV$:
This is trivially true when $p= q$ and immediate from (CP) otherwise.
Consequently, $\pi$ is a polarity of $\bPV$, which is null by its definition.
The set of null lines of $\pi$ is a general linear complex of lines \big(a
geometric hyperplane of $\bGV$\big) containing $\cS_1\cup\cS_2$. By virtue of
Theorem~\ref{thm:hyper}, the parallelism $\parallel$ is Clifford.
\end{proof}

\begin{rem}
The condition of crossed pencils can readily be translated to an
\emph{arbitrary}\/ three{\trenn}dimensional projective space $(\cP,\cL)$, and
it could be used to define when a parallelism of this space is \emph{Clifford}.
As in the proof of Theorem~\ref{thm:crossed}, the existence of such a Clifford
parallelism implies that the projective space $(\cP,\cL)$ admits a null
polarity, which in turn forces $(\cP,\cL)$ to be Pappian.
\end{rem}

\section{Conclusion}

By our investigation, which is far from being comprehensive, there is a one-one
correspondence between Clifford parallelisms and external planes to the Klein
quadric. We have not included several topics. Among these is the problem of
finding necessary and sufficient algebraic conditions for two parallel classes
of a given Clifford parallelism to be projectively equivalent, even though all
necessary tools can be found in the literature: In the setting from
Section~\ref{sect:Cliff-Klein} there are one-one correspondences among
(i)~quadratic extension fields of $F$ that are contained in $H$ (which are
precisely the lines of $\bPV$ through the point $F1$), (ii)~parallel classes of
${\lip}={\Cp}$, (iii)~solids through $C$, and (via $\perp$) (iv)~lines in the
plane $C^\perp$. More generally, any external line to the Klein quadric can be
linked directly with a quadratic extension field of $F$ and vice versa. We
refer to \cite{beut+u-93a}, \cite{cica+d+s-12a}, \cite{debr-10a},
\cite{debr-11a}, \cite{gulde+st-12a}, \cite{havl-94b}, \cite{havl-94a}, and
\cite{stro-08a} for a wealth of (overlapping) results that should settle the
issue. Also we have not incorporated the results from \cite{blunck+p+p-10a} and
\cite{havl-15a}, where Clifford parallelisms have been described by extending
the ground field. Finally, we are of the opinion that \emph{kinematic line
mappings}\/ and related work from \cite{brau-73a}, \cite{havl-83c},
\cite{lueb-79a}, and \cite{odeh-04a} could provide a good guideline for a
generalisation of our findings to other parallelisms arising from
\emph{kinematic spaces}; see \cite{karz+k-88a}, \cite{paso-10a}, and the
references therein.


\newcommand{\Dbar}{\makebox[0cm][c]{\hspace{-2.5ex}\raisebox{0.25ex}{-}}}
  \newcommand{\cprime}{$'$}

\noindent
Hans Havlicek\\
Institut f\"{u}r Diskrete Mathematik und Geometrie\\
Technische Universit\"{a}t\\
Wiedner Hauptstra{\ss}e 8--10/104\\
A-1040 Wien\\
Austria\\
\texttt{havlicek@geometrie.tuwien.ac.at}

\end{document}